\documentclass[reqno]{amsart}
\usepackage{tikz}
\usetikzlibrary{patterns} 
\usepackage{pgfplots}
\usepackage{amsfonts,amsmath,amsthm,amssymb,mathrsfs}
\usepackage{color}
\usepackage{amsmath,amssymb,amsfonts, bm}
\usepackage{graphicx}
\usepackage{newunicodechar}
\usepackage{xcolor}
\usepackage{geometry}
\numberwithin{equation}{section}
\usepackage[pdfstartview=FitH,colorlinks=true]{hyperref}
\hypersetup{urlcolor=red, linkcolor=blue,citecolor=red}
\allowdisplaybreaks

\theoremstyle{plain}

\usepackage{ctex}
\title{Well-posedness of inhomogeneous nonlinear wave equations in $\mathbb{R}^3$}
\newtheorem{thm}{Theorem}[section]
\newtheorem{definition}[thm]{Definition}

\newtheorem{lemma}[thm]{Lemma}
\newtheorem{remark}[thm]{Remark}
\renewcommand{\proofname}{Proof}

 \numberwithin{equation}{section}
\def \Rn {\mathbb{R}^n}

\begin{document}
\author[]
{Boyu Jiang, Jiawei Shen, Kexue Li*}

\address{Boyu Jiang
 \newline\indent
 School of Mathematics and Statistics, Xi'an Jiaotong University, Xi'an ,710049, Shaanxi,China
 \newline\indent
 Jiawei Shen
 \newline\indent
 School of Mathematics and Computational Science, Xiangtan
University, Xiangtan, 411105, Hunan, China
 \newline\indent
 Kexue Li
 \newline\indent
 School of Mathematics and Statistics, Xi'an Jiaotong University,Xi'an 710049, Shaanxi,China
 \newline\indent
 }
 \email{ jiangboyu@stu.xjtu.edu.cn; sjiawei633@xtu.edu.cn; kxli@mail.xjtu.edu.cn}
\date{}

\begin{abstract}
     This paper is devoted to the well-posedness of the inhomogeneous nonlinear wave equations. By combining Strichartz estimates with the contraction mapping principle, we establish local and global well-posedness in the function spaces $\dot{H}^1(\mathbb{R}^3)\times L^2(\mathbb{R}^3)$ and $\dot{H}^{s+1}(\mathbb{R}^3)\times \dot{H}^{s}(\mathbb{R}^3)$. The analysis is carried out in the energy-subcritical regime. As a consequence, our results extend and improve upon previous results in the literature for general nonlinear wave equations.
\end{abstract}


\maketitle

\section{Introduction}
In this article, we primarily concentrate on the well-posedness for inhomogeneous nonlinear wave equations in $H^1(\mathbb{R}^3)\times L^2(\mathbb{R}^3)$ and $\dot{H}^{s+1}(\mathbb{R}^3)\times \dot{H}^{s}(\mathbb{R}^3)$. We begin by presenting the form of the equation, 
\begin{equation}\label{1.1}
\begin{cases}
u_{tt} - \Delta u + |x|^{-b}|u|^{\alpha}u = 0, & (x,t) \in \mathbb{R}^3 \times \mathbb{R}, \\[4pt]
u(0) = \varphi(x), \quad u_t(0) = \psi(x), & x \in \mathbb{R}^3
\end{cases}.
\end{equation} 
This type of nonlinearity was first studied in \cite{Guzmán2016}, where the well-posedness of the Schrödinger equation with this nonlinear term was investigated. Motivated by this article, we investigate the well-posedness theory of inhomogeneous nonlinear wave equations. When $b=0$, the equation reduces to the standard nonlinear wave equation. Research on this equation traces back to Jörgen, who investigated the Cauchy problem for the semilinear wave equation.
\begin{equation} \label{1.2}
\begin{cases}
u_{tt} - \Delta u + f(u) = 0, & (x,t) \in \mathbb{R}^n \times \mathbb{R}, \\[4pt]
u(0) = \varphi(x), \quad u_t(0) = \psi(x), & x \in \mathbb{R}^n
\end{cases}
\end{equation}
where
\begin{equation}\label{1.3}
   f(u) = u|u|^{p-1}, \quad 1 < p \leqslant \frac{n+2}{n-2}, \quad n \geqslant 3. 
\end{equation}
where $p_c=\frac{n+2}{n-2}$ corresponds to the $H^1$ critical exponent for \eqref{1.2}.When $n \le 2$, all $1 < p < \infty$ lie in the subcritical growth range. For $n=3$ and $1<p<p_c=5$, Jörgen \cite{Jorgens1961} established the global well-posedness of smooth solutions to \eqref{1.2}, \eqref{1.3} in 1961. For higher dimensions $(3 < n \le 9)$, global well-posedness of smooth solutions was proved by Brenner and Wahl, as well as Pecher, in \cite{Pecher1976, Wolf1971}. Furthermore, Ginibre and Velo \cite{Ginibre1985, Ginibre1989} resolved the global well-posedness of energy solutions for \eqref{1.2} and \eqref{1.3} in the energy-subcritical regime. Additional well-posedness conclusions for the energy-critical regime exist, but we shall not discuss them in detail here. Recent years have witnessed further advances. For example, Shao et al. \cite{Shao2025} demonstrated the existence of blow-up solutions for the defocusing wave equation in the supercritical setting—a phenomenon previously deemed impossible in the energy-critical and energy-subcritical regimes. Additionally, for dimensions $n \ge 4$, Jendrej \cite{Jendrej2023} and coauthors demonstrated that radial ground state solutions exist for the nonlinear wave equation in the energy-critical setting. 
As is well known, the wave equation admits a unique solution in $C([0,T], L^2) \cap C^1([0,T], H^{-1})$ if and only if it satisfies the following integral equation—Duhamel's formula:
\[ 
u(x,t) = \mathcal{F}^{-1} \cos|\xi|t \,\mathcal{F}\varphi
+ \mathcal{F}^{-1} \frac{\sin|\xi|t}{|\xi|} \,\mathcal{F}\psi
+ \int_0^t \mathcal{F}^{-1} \frac{\sin|\xi|(t-\tau)}{|\xi|} \,\mathcal{F}h \,d\tau
\]

\[
\triangleq \dot{K}(t)\varphi + K(t)\psi
+ \int_0^t K(t-\tau)h(x,\tau)\,d\tau
\]

\[
\triangleq \dot{K}(t)\varphi + K(t)\psi + Gh
\quad\left(
K(t) = \mathcal{F}^{-1} \frac{\sin|\xi|t}{|\xi|} \,\mathcal{F}
\right)
\]
which holds for all $t \in [0,T]$. With further regularity assumptions on the nonlinear term and initial data—for instance, in three dimensions, $f \in C^\infty(\mathbb{R})$ and $f(0)=0$, along with $\phi, \psi \in C^\infty_c(\mathbb{R}^3)$—we have $u \in C([0,T], C^\infty_c(\mathbb{R}^3))$.\par
Moreover, the nonlinear wave equation admits the familiar energy conservation law, 
\begin{equation}
    E(u(t)) = \frac{1}{2}\int |u_t(t,x)|^2 dx + \frac{1}{2}\int |\nabla u(t,x)|^2 dx + \frac{1}{p+1}\int |u(t,x)|^{p+1} dx=E(u(0))
\end{equation}
We now present our main result:
\begin{thm} \label{local well-posedness L^2}
    Let $0 < \alpha < \frac{4-2b}{3}$ and $0 < b < 2$. Then, for any initial data $(\phi, \psi) \in \dot{H}^{1}(\mathbb{R}^3) \times L^{2}(\mathbb{R}^3)$, equation \eqref{1.1} is locally well-posed in the space $\dot{H}^1(\mathbb{R}^3) \times L^2(\mathbb{R}^3)$.
\end{thm}
\begin{thm} \label{global well-posedness L^2}
    Under the conditions $0 < \alpha < \frac{4-2b}{3}$ and $0 < b < 2$, equation \eqref{1.1} is globally well-posed in $\dot{H}^1(\mathbb{R}^3) \times L^2(\mathbb{R}^3)$ for small initial data.
\end{thm}
\begin{thm}\label{local H^s}
    Let $0<\alpha<\frac{4-2b}{3-2s}$ and $\frac{1}{2}<b<\frac{3}{2}$. If $0<s<b-\frac{1}{2}$.  Then for any initial data $(\varphi,\psi)\in\dot{H}^{s+1}(\mathbb{R}^3)\times\dot{H}^{s}(\mathbb{R}^3)$, equation \eqref{1.1} is locally well-posed in the space $\dot{H}^{s+1}(\mathbb{R}^3)\times\dot{H}^{s}(\mathbb{R}^3)$. 
\end{thm}
\begin{remark}
   A global existence result for solutions under the conditions of Theorem \ref{local H^s} is not obtained in the present paper. We expect to improve this in future work.
\end{remark}

This paper is organized as follows. In Section 2, we collect some necessary definitions and preliminary results. The proofs of Theorems \ref{local well-posedness L^2} and \ref{global well-posedness L^2} are carried out in Section 3, and Section 4 is devoted to the proof of Theorem \ref{local H^s}.
\section{Basic Results and Preliminaries}
The definitions and theorems presented in this section are based on \cite{Miao2005} and \cite{Guzmán2016}.

Let $S(\Rn)$ be the set of the functions such that 
\[
S(\Rn):=\{f\in C^\infty(\Rn):|x|^\beta|\partial^\alpha
_x(f)|<\infty \quad \forall \alpha,\beta\in \mathbb{N}^n\},
\] 
and let $S'(\Rn)$ be the dual space of $S(\Rn).$ Now, to be convenient, we introduce the Littewood-Paley decomposition. Let $\phi$ be smooth function supported in $B(0,2)$ and $\phi\equiv 1$ when $|\xi|\le 1.$  So, for dyadic $2^N\in 2^\mathbb{Z},$  we then define $P_{N}$
   \begin{align*}
        &\widehat{P_{ N} f}(\xi)=(\phi(\xi/N)-\phi(2\xi/N))\hat{f}(\xi).
    \end{align*}
We introduce the defination of Besov Space.
\begin{definition}
Define the Besov and Triebel spaces separately. 
 \[
 \dot{B}_{p,q}^{\sigma} = \left\{ f \in S'(\Rn)/\mathcal{P},\ \left\{ \sum_{j} 2^{j\sigma q} \| P_{ N} f\|_{L^p}^q \right\}^{1/q} = \|f\|_{\dot{B}_{p,q}^{\sigma}} < \infty \right\},
 \]
 \[
 \dot{F}_{p,q}^{\sigma} = \left\{ v \in S'(\Rn)/\mathcal{P},\ \left\| \left\{ \sum_j |2^{j\sigma} P_Nf|^q \right\}^{1/q} \right\|_p = \|v\|_{\dot{F}_{p,q}^{\sigma}} < \infty \right\}.
 \]
where $\mathcal{P}$ is the set of all polynomial on $\Rn$   

\end{definition}
For the Besov Space, we have the following properties.
\begin{lemma}
    令$1<p,q<\infty$, $\sigma\in \mathbb{R}$, There exist the following embedding relations
    \begin{equation}
        \dot{B}^\sigma_{p,\min(p,q)}\hookrightarrow \dot{F}^\sigma_{p,q}\hookrightarrow \dot{B}^\sigma_{p,\max(p,q)}.
    \end{equation}
\end{lemma}
\begin{lemma}
let $1 \leqslant p \leqslant r \leqslant \infty,$ and $\quad 1 \leqslant q \leqslant \infty.$ Then, for $ \rho, \sigma \in \mathbb{R},$ we have 
\[
\dot{B}_{p,q}^{\sigma} \hookrightarrow \dot{B}_{r,q}^{\rho}, \quad\text{such that} \quad\|f\|_{ \dot{B}_{r,q}^{\rho}} \leqslant C \|f\|_{\dot{B}_{p,q}^{\sigma}}, 
\] when $\sigma - \dfrac{n}{p} = \rho - \dfrac{n}{r}.$
   
\end{lemma}
\begin{definition}\cite{Miao2005}
  We call the following pair of indices a wave-admissible pair if
$2 \le q, r \le \infty$, $(q, r, \gamma(r)) \neq (2, \infty, 1)$,
and it satisfies
    $$0\le\frac{2}{q}\le\gamma(r)=(n-1)(\frac{1}{2}-\frac{1}{r})\le1.$$
    We denote $(q, r) \in \tilde{\mathcal{A}}$. In particular, when
$\frac{2}{q} = \gamma(r)$, the pair $(q, r)$ is called an optimal wave-admissible pair, denoted $(q, r) \in \mathcal{A}$.
\end{definition}
In this paper, we define the Strichartz norm associated with the wave equation by
$$\|u\|_{W(\dot{H}^s)}=\sup_{(q,r)\in\mathcal{A}} \|u\|_{L^q_tH^{s,3r}_x}$$
and denote by the dual Strichartz norm
$$\|u\|_{W^{\prime}(\dot{H}^{-s})}=\inf_{(q,r)\in\mathcal{A}} \|u\|_{L^{q^\prime}_tH^{s,3r^\prime}_x}.$$
\begin{remark}
    When $s = 0$, $W(\dot{H}^0) = W(L^2)$. Here, both $W(\dot{H}^s)$ and $W'(\dot{H}^{-s})$ denote estimates over the entire Euclidean space $\mathbb{R} \times \mathbb{R}^N$. For a time interval $I \subset (-\infty, \infty)$ and a subset $A \subset \mathbb{R}^N$, we similarly adopt the notation $W(\dot{H}^s(A); I)$ and $W'(\dot{H}^{-s}(A); I)$.
\end{remark} 
\begin{thm}[Strichartz estimates for the wave equation]\cite{Miao2005} \label{strichartz}
Suppose $(q, r), (q_j, r_j) \in \mathcal{A}\ (j=1,2)$, and $I = \mathbb{R}$ or $I \subset \mathbb{R}$ with $0 \in \overline{I}$. Then the following estimates hold:
\[
\|\dot{K}\varphi\|_{L^q(I;\dot{B}_{r,2}^{\sigma-\beta(r)})}
+ \|K(t)\psi\|_{L^q(I;\dot{B}_{r,2}^{\sigma-\beta(r)})}
\leq C\,\|(D^\sigma\varphi,D^{\sigma-1}\psi)\|_2,
\]
\[
\|Gh\|_{L^{q_1}(I;\dot{B}_{r_1,2}^{\sigma-\beta(r_1)})}
\leq C\,\|h\|_{L^{q_2'}(I;\dot{B}_{r_2',2}^{\sigma+\beta(r_2)-1})},
\]
Particularly,
\[
\|Gh\|_{L^q(I;\dot{B}_{r,2}^{\sigma-\beta(r)})}
\leq C\,\|h\|_{L^{q'}(I;\dot{B}_{r',2}^{\sigma+\beta(r)-1})}.
\]
\end{thm}
\begin{remark}
    In the application of Strichartz estimates, it is sufficient to prove them only for the optimal wave admissible pairs, For the proof, see \cite{Miao2005}.
\end{remark}

To study the well-posedness of nonlinear wave equations in $\dot{H}^{s+1}\times \dot{H}^s$(with $s\in(0,1]$) we need to introduce the following two lemmas, both taken from \cite{Guzmán2016}.  
\begin{lemma}(Fractional product rule)  \label{Fractional}
    let $s\in (0,1]$ and $1<r,r_1,r_2,p_1,p_2<\infty$, satisfy $\frac{1}{r}=\frac{1}{r_i}+\frac{1}{p_i}$, for $\forall$ i=1,2. Then, 
    $$\|D^s(fg)\|_{L^r}\le c\|f\|_{L^{r_1}}\|D^sg\|_{L^{p_1}}+c\|D^sf\|_{L^{r_2}}\|g\|_{L^{p_2}}$$.
\end{lemma}
\begin{lemma}(Fractional chain rule) \label{chain}
    Assume that $G\in C^1(\mathbb{C})$, $s\in (0,1]$, and $1<r,r_1,r_2<\infty$, satisfy $\frac{1}{r}=\frac{1}{r_1}+\frac{1}{r_2}$. Then,
    $$\|D^sG(u)\|_{L^r}\lesssim \|G^{\prime}(u)\|_{L^{r_1}}\|D^su\|_{L^{r_2}}$$ 
\end{lemma}
\subsection{Some notations and remarks}
Throughout this paper, $C$ is a postive finite constant which is independent of  the essential variables. $ A\lesssim B$ mean $A\leq CB$ for some constant. If both $A\lesssim B$ and $B\lesssim A,$ then we donate $A\sim B.$  We will use $ X\hookrightarrow Y$ to mean a continuous embedding, that, for two function spaces $X,Y,$  there an inclusion map $X\rightarrow Y$ with $\|f\|_Y\lesssim\|f\|_X.$
\section{the proof of the main theory: part 1}
Prior to the proof of the main theorem, we state a key lemma regarding the estimation of the nonlinearity.
\begin{lemma} \label{nonlinear}
    Given $0 < \alpha < \frac{4-2b}{3}$ and $0 < b < 2$, we have the following estimate for the nonlinear term:
\begin{equation}\label{nonlinear estimate}
\||x|^{-b} |u|^{\alpha} u\|_{L^1(I; L^2)} \lesssim (T^{\theta_1} + T^{\theta_2}) \|u\|^{\alpha+1}_{W(I; L^2)}
\end{equation}
where $I = [0, T]$, and $\theta_1, \theta_2 > 0$.
\end{lemma}
\begin{proof}
    \begin{equation}\label{3.2}
    \begin{aligned}
         \||x|^{-b}|u|^{\alpha}u\|_{L^2}&\le \||x|^{-b}|u|^{\alpha}u\|_{L^2(B)}+\||x|^{-b}|u|^{\alpha}u\|_{L^2(B^c)} \\
         &\le \||x|^{-b}|u|^{\alpha}u\|_{L^2(B)}+\||u|^{\alpha}u\|_{L^2}
    \end{aligned}
    \end{equation}
    $\||u|^{\alpha}u\|_{L^2}=\|u\|_{2(\alpha+1)}^{\alpha+1} $, therefore, by Hölder's inequality,
    \begin{equation}\label{3.3}
      \|\|u\|_{2(\alpha+1)}^{\alpha+1}\|_{L^1}\lesssim T^{\frac{4-\alpha}{2}}\|u\|^{\alpha+1}_{L^{\frac{2(\alpha+1)}{\alpha-2}}(I, L^{2(\alpha+1)})} \le T^{\theta_1}\|u\|_{W(L^2; I)}^{\alpha+1},   
    \end{equation}
    where $\theta_1=\frac{4-\alpha}{2}.$
    One can readily verify that $(\frac{2(\alpha+1)}{\alpha-2}, \frac{2(\alpha+1)}{3})$ constitutes an optimal wave-admissible pair. Meanwhile
    \begin{equation}\label{3.4}
        \||x|^{-b}|u|^{\alpha}u\|_{L^2(B)}\le \||x|^{-b}\|_{L^{\gamma}(B)}\||u|^{\alpha}u\|_{L^{\frac{2\gamma}{\gamma-2}}}\le \||x|^{-b}\|_{L^{\gamma}(B)}\|u\|^{\alpha+1}_{L^{\frac{2\gamma(\alpha+1)}{\gamma-2}}}.
    \end{equation}
    So
    \begin{equation} \label{3.5}
      \|\|u\|^{\alpha+1}_{L^{\frac{2\gamma(\alpha+1)}{\gamma-2}}}\|_{L^1(I)}\lesssim T^{\frac{\alpha\gamma+4\gamma-6}{2\gamma(\alpha+1)}}\|u\|_{L^{\frac{2\gamma(\alpha+1)}{(\alpha-2)\gamma+6}}(I, L^{\frac{2\gamma(\alpha+1)}{\gamma-2}})}\le T^{\theta_2}\|u\|_{W(L^2; I)}  
    \end{equation}
    where $\theta_2=\frac{\alpha\gamma+4\gamma-6}{2\gamma(\alpha+1)}$, It is also easy to verify that $\left( \frac{2\gamma(\alpha+1)}{(\alpha-2)\gamma + 6}, \frac{2\gamma(\alpha+1)}{3(\gamma - 2)} \right)$ is an optimal wave-admissible pair. From \eqref{3.2}, \eqref{3.3}, \eqref{3.4}, and \eqref{3.5}, we derive \eqref{nonlinear estimate}, which concludes the proof of Lemma \ref{nonlinear}.
\end{proof}
\begin{remark}
    Indeed, there exists a more general expression, 
\begin{equation}
     \||x|^{-b}|u|^{\alpha}v\|_{L^1(I;L^2)}\lesssim (T^{\theta_1}+T^{\theta_2})\|u\|^{\alpha}_{W(I;L^2)}\|v\|_{W(I;L^2)},
\end{equation}
The proof is similar. Hence, we omit the proof here.
\end{remark} 

Armed with this lemma, we are now in a position to prove Theorem \ref{local well-posedness L^2}, 
{
\renewcommand{\proofname}{Proof of Theorem \ref{local well-posedness L^2}}
\begin{proof}
    By virtue of the time reversal symmetry of the wave equation, it suffices to consider positive time. We begin by defining
    \begin{equation}
        X=C([0,T]; \dot{H}^1(\mathbb{R}^3))\cap C^1([0,T];L^2(\mathbb{R}^3))\cap W([0,T];L^2(\mathbb{R}^3))
    \end{equation}
    and we fix
    $$B(a,T)=\{u\in X; \|u\|_{W([0,T];L^2)}\le a \}, $$
    Here, $a$ and $T$ are two undetermined positive constants. We define the mapping
    \begin{equation}
        H(u)(t)=\dot{K}(t)\varphi + K(t)\psi
+ \int_0^t K(t-\tau)[|x|^{-b}|u|^{\alpha}u](\tau) \mathrm{d}\tau
    \end{equation}
    We now demonstrate that this map is a contraction mapping on $B(a,T)$.
    By applying the Strichartz estimates (Theorem \ref{strichartz}) in conjunction with Lemma \ref{nonlinear}, we have
    \begin{equation*}
    \begin{aligned}
                \|H(u)\|_{L^q(I; L^{3r})}&\le \|\dot{K}(t)\varphi + K(t)\psi\|_{L^q(I;L^{3r})}+ \|\int_0^t K(t-\tau)[|x|^{-b}|u|^{\alpha}u](\tau) \mathrm{d}\tau\|_{L^q(I;L^{3r})}\\
                &\le \|\dot{K}(t)\varphi + K(t)\psi\|_{L^q(I;B^{\frac{2}{r}}_{r,2})}+\|\int_0^t K(t-\tau)[|x|^{-b}|u|^{\alpha}u](\tau) \mathrm{d}\tau\|_{L^q(I;B^{\frac{2}{r}}_{r,2})}\\
                &\lesssim\|\varphi\|_{\dot{H}^1}+\|\psi\|_{L^2}+\||x|^{-b}|u|^{\alpha}u\|_{L^1(I;L^2)} \\
                &\lesssim \|\varphi\|_{\dot{H}^1}+\|\psi\|_{L^2}+(T^{\theta_1}+T^{\theta_2})\|u\|_{W(I;L^2)}^{\alpha+1}.
    \end{aligned}
    \end{equation*}  
    Taking the supremum over all optimal wave-admissible pairs on both sides yields, $$\|H(u)\|_{W(I;L^2)}\lesssim \|\varphi\|_{\dot{H}^1}+\|\psi\|_{L^2}+(T^{\theta_1}+T^{\theta_2})a^{\alpha+1}.$$
    Choosing $a = 2c\|u_0\|_{L^2}$ and taking $T$ sufficiently small so that
    $$ca^{\alpha}(T^{\theta_1}+T^{\theta_2})\le \frac{1}{2}, $$ we obtain that $H$ is a self-map on $B(a,T)$. 
    It now suffices to prove that $H$ is a contraction mapping. 
   Applying the Strichartz estimates once more, for any wave-admissible pair $(q, r)$, we have
    \begin{equation}
    \begin{aligned}
           \|H(u)-H(v)\|_{L^q(I;L^{3r})}&\lesssim \||x|^{-b}(|u|^{\alpha}u-|v|^{\alpha}v)\|_{L^1(I;L^2)}\\
           &\lesssim \||x|^{-b}|u|^{\alpha}|u-v||\|_{L^1(I;L^2)}+\||x|^{-b}|v|^{\alpha}|u-v|\|_{L^1(I;L^2)}\\
           &\lesssim(T^{\theta_1}+T^{\theta_2})(\|u\|^{\alpha}_{W(I;L^2)}+\|v\|^{\alpha}_{W(I;L^2)})\|u-v\|_{W(I;L^2)}\\
           &\lesssim (T^{\theta_1}+T^{\theta_2})a^{\alpha}\|u-v\|_{W(I;L^2)},
    \end{aligned}
    \end{equation}
    where $u, v\in B(a,T)$. Similarly, we can choose $a$ and $T$ such that the preceding coefficient is less than $\frac{1}{2}$.
    Thus we obtain
    \begin{equation}
           \|H(u)-H(v)\|_{W(I;L^2)}
           \lesssim \frac{1}{2}\|u-v\|_{W(I;L^2)},
    \end{equation}
    That is, $H$ is a contraction mapping on $B(a,T)$, which completes the proof of the theorem.
\end{proof}
}
{
\renewcommand{\proofname}{Proof of Theorem \ref{global well-posedness L^2}}
\begin{proof}
    From local well-posedness together with energy conservation, we deduce global well-posedness in the sense of small initial data. In fact, the Gagliardo–Nirenberg interpolation inequality yields
    \begin{equation*}
    \begin{aligned}
          E(u(0))=&\|\psi\|^2_{L^2(\mathbb{R}^3)}+\|\nabla\varphi\|^2_{L^2(\mathbb{R}^3)}+\frac{1}{p+1}\|\varphi\|^{p+1}_{L^{p+1}(\mathbb{R}^3)}\\
        &\lesssim \|\psi\|^2_{L^2(\mathbb{R}^3)}+\|\nabla\varphi\|^2_{L^2(\mathbb{R}^3)}+\|\nabla\varphi\|_{L^2(\mathbb{R}^3))}^2\|\varphi\|_{L^{\frac{3(p-1)}{2}}}^{p-1} \\
        &\lesssim \|\psi\|^2_{L^2(\mathbb{R}^3)}+\|\nabla\varphi\|^2_{L^2(\mathbb{R}^3)}+\|\nabla\varphi\|_{L^2(\mathbb{R}^3))}^2\|\varphi\|_{\dot{H}^{\frac{3}{2}-\frac{2}{p-1}}}^{p-1},
    \end{aligned}  
    \end{equation*}
    thus
    \begin{equation}
    E(u(t))\lesssim\|\psi\|^2_{L^2(\mathbb{R}^3)}+\|\nabla\varphi\|^2_{L^2(\mathbb{R}^3)}+\|\nabla\varphi\|_{L^2(\mathbb{R}^3))}^2\|\varphi\|_{\dot{H}^{\frac{3}{2}-\frac{2}{p-1}}}^{p-1}\lesssim \|\psi\|^2_{L^2(\mathbb{R}^3)}+\|\nabla\varphi\|^2_{L^2(\mathbb{R}^3)}
    \end{equation}
    holds for all $t$, and for any $\psi \in L^2$ and $\varphi \in \dot{H}^1 \cap \dot{H}^{\frac{3}{2} - \frac{2}{p-1}}$. By a density argument, it follows that $E(u)$ has an a priori bound for any initial data in $H^1 \times L^2$. Then, by the blow-up alternative, we conclude that $u$ is globally well-posed. This completes the proof, which is similar to that in \cite{Dodson2024}.
\end{proof}
}
\section{the proof of the main theory: part 2}
We now focus on establishing the well-posedness in the space $H^{s+1}\times H^s$. The core of the proof still lies in estimating the nonlinear terms. Before presenting the estimates, we have the following remark., 
\begin{remark}
    Under the conditions of Lemma \ref{nonlinear2} below, we assert that $D^s(|x|^{-b})=C_{N,b}|x|^{-b-s}$, for the proof, see \cite{Guzmán2016}.
\end{remark} 

The lemma for estimating the nonlinear term is given below.
\begin{lemma} \label{nonlinear2}
    Let $\frac{1}{2}<b<\frac{3}{2}$, if $0<s<b-\frac{1}{2}$ and $0<\alpha<\frac{4-2b}{3-2s}$, then we have the following estimate for the nonlinear term:
    \begin{equation} \label{4.1}
        \|D^s(|x|^{-b}|u|^{\alpha}u)\|_{L^1(I; L^2)}\lesssim(T^{\theta_1}+T^{\theta_2})\|u\|_{W(I; \dot{H}^s)}^{\alpha+1}.
    \end{equation}
    Here $I=[0,T]$, and $\theta_1,\theta_2>0$. 
\end{lemma}
\begin{proof}
    We take $B=B(0,1)$, to denote the unit ball in $\mathbb{R}^n$. Then we have
    \begin{equation*}
        \|D^s[|x|^{-b}|u|^{\alpha}u]\|_{L^1(I;L^2)}\le \|D^s[|x|^{-b}|u|^{\alpha}u]\|_{L^1(I; L^2(B^c))}+\|D^s[|x|^{-b}|u|^{\alpha}u]\|_{L^1(I; L^2(B))}.
    \end{equation*}
    Next, we estimate the second term on the right-hand side. Combining the previous lemmas \ref{Fractional} and \ref{chain}, and using Hölder's inequality, we obtain
    \begin{equation*}
    \begin{aligned}
         \|D^s[|x|^{-b}|u|^{\alpha}u]\|_{L^2(B)}&\le \|D^s|x|^{-b}\|_{L^{r_1}(B)}\||u|^{\alpha}u\|_{L^{p_1}}+\||x|^{-b}\|_{L^{r_2}(B)}\|D^s[|u|^{\alpha}u]\|_{L^{p_2}}\\
         &\lesssim\|D^s|x|^{-b}\|_{L^{r_1}(B)}\|u\|_{L^{(\alpha+1)p_1}}^{\alpha+1}+\||x|^{-b}\|_{L^{r_2}(B)}\||u|^{\alpha}D^su\|_{L^{p_2}}\\
         &\lesssim \|D^s|x|^{-b}\|_{L^{r_1}(B)}\|u\|_{L^{(\alpha+1)p_1}}^{\alpha+1}+\||x|^{-b}\|_{L^{r_2}(B)}\||u|^{\alpha}\|_{L^{\frac{1+\frac{1}{\alpha}}{\frac{1}{p_2}-\frac{s}{3}}}}\|D^su\|_{L^{\frac{1+\frac{1}{\alpha}}{\frac{1}{\alpha p_2}+\frac{s}{3}}}}\\
         &\lesssim \|D^s|x|^{-b}\|_{L^{r_1}(B)}\|u\|_{L^{(\alpha+1)p_1}}^{\alpha+1}+\||x|^{-b}\|_{L^{r_2}(B)}\|D^su\|^{\alpha+1}_{L^{\frac{1+\frac{1}{\alpha}}{\frac{1}{\alpha p_2}+\frac{s}{3}}}}\\
    \end{aligned}
    \end{equation*}
    Here we use the Sobolev embedding, where we fix $p_2$ to be a number greater than $\max\{\frac{6}{\alpha+1-2s\alpha},\frac{6}{3-2b}\}$. In this case, $r_2=\frac{1}{\frac{1}{2}-\frac{1}{p_2}}<\frac{3}{b}$, and thus$\||x|^{-b}\|_{L^{\frac{1}{\frac{1}{2}-\frac{1}{p_2}}}}< \infty$. 
    Therefore, we use the Hölder's inequality to deduce that $$\|\||x|^{-b}\|_{L^{\frac{1}{\frac{1}{2}-\frac{1}{p_2}}}}\|u\|^{\alpha+1}_{\dot{H}^{s,\frac{1+\frac{1}{\alpha}}{\frac{1}{\alpha p_2}+\frac{s}{3}}}}\|_{L^1(I)}\lesssim T^{\theta_1}\|u\|_{L^{\frac{2(1+\frac{1}{\alpha})}{1+\frac{1}{\alpha}-\frac{6}{\alpha p_2}-2s}}(I,\dot{H}^{s,\frac{1+\frac{1}{\alpha}}{\frac{1}{\alpha p_2}+\frac{s}{3}}})}^{\alpha+1}, $$
    and it is easy to verify that $(\frac{2(1+\frac{1}{\alpha})}{1+\frac{1}{\alpha}-\frac{6}{\alpha p_2}-2s},\frac{1+\frac{3}{\alpha}}{\frac{1}{\alpha p_2}+s})$ is the optimal admissible pair. Hence we have
    \begin{equation} \label{4.2}
\| \||x|^{-b}\|_{L^{\frac{1}{\frac{1}{2}-\frac{1}{p_2}}}(B)}\|u\|^{\alpha+1}_{\dot{H}^{s,\frac{1+\frac{1}{\alpha}}{\frac{1}{\alpha p_2}+\frac{s}{3}}}} \|_{L^1(I)} \lesssim T^{\theta_1} \|u\|^{\alpha+1}_{W(I; \dot{H}^s)}.
\end{equation}
Similarly, we obtain
\begin{equation} \label{4.3}
\| \|D^s|x|^{-b}\|_{L^{r_1}(B)}\|u\|_{L^{(\alpha+1)p_1}}^{\alpha+1} \|_{L^1(I)} \lesssim T^{\theta_2} \|u\|^{\alpha+1}_{W(I; \dot{H}^s)} .
\end{equation}
Combining \eqref{4.2} and \eqref{4.3}, we arrive at the estimate \eqref{4.1}, thus completing the proof of the lemma. 
\end{proof}
\begin{remark}\label{REMARK2}
     Similarly, for the general nonlinear term $|x|^{-b}|u|^{\alpha}v$, the following estimate can also be made.
\begin{equation}
    \||x|^{-b}|u|^{\alpha}v\|_{L^1(I, L^2)}\lesssim (T^{\theta_1}+T^{\theta_2})\|u\|^{\alpha}_{W(I;\dot{H}^s)}\|v\|_{W(I; L^2)},
\end{equation}
The proof is completely analogous.
\end{remark}
We now begin the proof of Theorem \ref{local H^s}.
{
\renewcommand{\proofname}{Proof of Theorem \ref{local H^s}}
\begin{proof}
    By using the time symmetry of the wave equation, we only need to consider forward time. We define $X=C([0,T]; \dot{H}^{s+1})\cap C^1([0,T];\dot{H}^s)\cap W([0,T]; \dot{H}^s)$, denote $I=[0,T]$, and define
    \begin{equation*}
        \|u\|_{T}=\|u\|_{W(I;L^2)}+\|u\|_{W(I;\dot{H}^s)}.
    \end{equation*}
    Following the same notation as in the proof of Theorem \ref{local well-posedness L^2}, we define a map $H(u)$. Set $S(a,T)=\{u\in X: \|u\|_{T}\le a\}$, where $a,T$ are two positive constants to be determined, and endow $S(a,T)$ with the metric.
    \begin{equation*}
        d_T(u,v)=\|u-v\|_{W(I; L^2)}.
    \end{equation*}
   We now prove that $S(a,T)$ is a complete metric space. The proof is essentially the same as that in \cite{Cazenave2003} (see Theorem 1.2.5 and the proof of Theorem 4.4.1 on page 94). Since $S(a,T)\subset X$, and $X$ is complete, it suffices to show that $S(a,T)$ is closed in $X$. Take a sequence $u_n \in S(a,T)$, with $u_n\rightarrow u$ as $n\rightarrow \infty$, we need to verify that $u\in S(a,T)$. Because $u_n \in C([0,T];\dot{H}^{s+1}(\mathbb{R}^3))$,  the set ${u_n(t)}$ is bounded in $\dot{H}^{s+1}(\mathbb{R}^3)$ for each $t \in I$. By the reflexivity of $\dot{H}^{s+1}$, there exists a weakly convergent subsequence, still denoted by $u_n(t)$, such that $u_n(t)\rightharpoonup v(t)$ in $\dot{H}^{s+1}(\mathbb{R}^3)$, and
    \begin{equation*}
        \|v(t)\|_{\dot{H}^{s+1}}\le \liminf_{n\rightarrow \infty}{\|u_n\|_{\dot{H}^{s+1}}}\le a
    \end{equation*}
    On the other hand, by our assumption $d_T(u_n, u) \to 0$, it follows that $u_n \to u$ in $L^q(I; L^{3r})$, where $(q, r)$ is any optimal wave admissible pair. Since $(\infty, 2)$ is also an optimal wave admissible pair, we obtain $u_n \to u$ in $L^\infty(I; L^6)$. Thus, for almost every $t \in I$, we have $u_n(t) \to u(t)$ in $L^6$. Taking the space $\dot{H}^{s+1} \cap L^6$, by uniqueness of the limit in this space, we obtain $u(t) = v(t)$. Then, using a density argument and the properties of $v$ itself, we obtain $\|u(t)\|_{\dot{H}^s_x} \le a$. That is, $u \in C(I; \dot{H}^{s+1}(\mathbb{R}^3))$. By similar reasoning, we can also show that $u \in W(I; \dot{H}^s)$ and $C^1(I; \dot{H}^s)$, and furthermore $\|u\|_T \le a$, thereby completing the proof of closedness.

    Next, applying the Strichartz estimates again, we have the following two estimates:
\begin{equation*}
\begin{aligned}
\|H(u)\|_{W(I; \dot{H}^s)} &\lesssim \|\varphi\|_{\dot{H}^{s+1}}+\|\psi\|_{\dot{H}^s}+\|D^s[|x|^{-b}|u|^{\alpha}u]\|_{L^1(I; L^2)},\\
\|H(u)\|_{W(I; L^2)} &\lesssim \|\varphi\|_{\dot{H}^{1}}+\|\psi\|_{L^2}+\|[|x|^{-b}|u|^{\alpha}u]\|_{L^1(I; L^2)}.
\end{aligned}
\end{equation*}
Using Lemma \ref{nonlinear2} and Remark \ref{REMARK2}, we have
\begin{equation*}
\|H(u)\|_T \lesssim \|\varphi\|_{\dot{H}^{s+1}}+\|\psi\|_{\dot{H}^s}+(T^{\theta_1}+T^{\theta_2})a^{\alpha+1}.
\end{equation*}
Similarly, we choose $a = 2\|(\varphi,\psi)\|_{\dot{H}^1\times L^2}$ and $T>0$ such that
$$a^{\alpha}(T^{\theta_1}+T^{\theta_2})<\frac{1}{4},$$
then we obtain that $H$ is a self-mapping on $S(a,T)$.

   On the other hand, by an argument similar to the previous one, we have
\begin{equation} \label{4.7}
d_T(H(u),H(v)) \lesssim (T^{\theta_1}+T^{\theta_2})(\|u\|^\alpha_T+\|v\|^\alpha_T) d_T(u,v).
\end{equation}
Taking $u,v\in S(a,T)$, from \eqref{4.7} we obtain
$$ d_T(H(u),H(v)) \lesssim (T^{\theta_1}+T^{\theta_2}) a^\alpha d_T(u,v) \le \frac{1}{4} d_T(u,v), $$
which implies that $H$ is a contraction mapping on $S(a,T)$. Applying the Banach fixed point theorem, there exists a unique fixed point in $S(a,T)$ such that the Duhamel formula holds. This completes the proof.
\end{proof}
}

\end{document}